\documentclass[11pt]{amsart}
\usepackage{amssymb,amsmath,amsthm}
\newtheorem{theorem}{Theorem}[section]
\newtheorem{proposition}[theorem]{Proposition}
\newtheorem{corollary}[theorem]{Corollary}
\newtheorem{lemma}[theorem]{Lemma}

\begin{document}

\title[Minkowski inequality]{A Minkowski inequality for hypersurfaces in the Anti-deSitter-Schwarzschild manifold}
\author{Simon Brendle, Pei-Ken Hung, and Mu-Tao Wang}
\begin{abstract}
We prove a sharp inequality for hypersurfaces in the $n$-dimensional Anti-deSitter-Schwarzschild manifold for general $n\geq 3$. This inequality generalizes the classical Minkowski inequality for surfaces in the three dimensional Euclidean space, and has a natural interpretation in terms of the Penrose inequality for collapsing null shells of dust. The proof relies on a new monotonicity formula for inverse mean curvature flow, and uses a geometric inequality established by the first author in \cite{Brendle}.
\end{abstract}
\address{Department of Mathematics \\ Stanford University \\ Stanford, CA 94305}
\address{Department of Mathematics \\ Columbia University \\ 2990 Broadway \\ New York, NY 10027}
\address{Department of Mathematics \\ Columbia University \\ 2990 Broadway \\ New York, NY 10027}
\thanks{The first author was supported in part by the National Science Foundation under grants DMS-0905628 and DMS-1201924. The third author was supported by the National Science Foundation under grant DMS-1105483. Part of this work was carried out while the third author was visiting the Taida Institute for Mathematical Sciences in Taipei, Taiwan.}

\maketitle

\section{Introduction}

The classical Minkowski inequality for a closed convex surface $\Sigma$ in $\mathbb{R}^3$ states that
\[\int_\Sigma H \, d\mu \geq \sqrt{16 \pi \, |\Sigma|},\] 
where $H$ is the mean curvature (i.e. the trace of the second fundamental form) and $|\Sigma|$ denotes the area of $\Sigma$ (cf. \cite{Minkowski}). For a convex hypersurface $\Sigma$ in $\mathbb{R}^n$, we have
\[\int_\Sigma H \, d\mu \geq (n-1) \, |S^{n-1}|^{\frac{1}{n-1}} \, |\Sigma|^{\frac{n-2}{n-1}}.\]
This was generalized to a mean convex and star-shaped surface using the method of inverse mean curvature flow (cf. \cite{Guan-Li}, \cite{Guan-Ma-Trudinger-Zhu}). Huisken recently showed that the assumption that $\Sigma$ is star-shaped can be replaced by the assumption that $\Sigma$ is outward-minimizing. Gallego and Solanes \cite{Gallego-Solanes} have obtained a generalization of Minkowski's inequality to the hyperbolic three space; however, this result does not seem to be sharp.

In this paper, we extend Minkowski's inequality to the case of surfaces in the Anti-deSitter Schwarzschild manifold. Let us recall the definition of the Anti-deSitter-Schwarzschild manifold. We fix a real number $m > 0$, and let $s_0$ denote the unique positive solution of the equation $1 + s_0^2 - m \, s_0^{2-n} = 0$. We then consider the manifold $M = S^{n-1} \times [s_0,\infty)$ equipped with the Riemannian metric
\[\bar{g} = \frac{1}{1 - m \, s^{2-n} + s^2} \, ds \otimes ds + s^2 \, g_{S^{n-1}},\]
where $g_{S^{n-1}}$ is the standard round metric on the unit sphere $S^{n-1}$. The sectional curvatures of $(M,\bar{g})$ approach $-1$ near infinity, so $\bar{g}$ is asymptotically hyperbolic. Moreover, the scalar curvature of $(M,\bar{g})$ equals $-n(n-1)$. The boundary $\partial M = S^{n-1} \times \{s_0\}$ is referred to as the horizon.

The Anti-deSitter Schwarzschild spaces are examples of static spaces. If we define
\begin{equation}
\label{static_potential}
f = \sqrt{1 - m \, s^{2-n} + s^2},
\end{equation}
then the function $f$ satisfies
\begin{equation}
\label{static_metric}
(\bar{\Delta} f) \, \bar{g}-\bar{D}^2 f+f \, \text{\rm Ric}=0.
\end{equation}
Taking the trace in \eqref{static_metric} gives $\bar{\Delta} f = nf$.

In general, a Riemannian metric is called static if it satisfies \eqref{static_metric} for some positive function $f$. The condition \eqref{static_metric} guarantees that the Lorentzian warped product $-f^2 \, dt \otimes dt+\bar{g}$ is a solution of Einstein's equations.

We now state the main result of this paper:

\begin{theorem}
\label{thm.a}
Let $\Sigma$ be a compact mean convex, star-shaped hypersurface $\Sigma$ in the AdS-Schwarzschild space, and let $\Omega$ denote the region bounded by $\Sigma$ and the horizon $\partial M$. Then
\begin{align*}
&\int_\Sigma f \, H \, d\mu - n(n-1) \int_\Omega f \, d\text{\rm vol} \\
&\geq (n-1) \, |S^{n-1}|^{\frac{1}{n-1}} \, \big ( |\Sigma|^{\frac{n-2}{n-1}} - |\partial M|^{\frac{n-2}{n-1}} \big ).
\end{align*}
Moreover, equality holds if and only if $\Sigma$ is a coordinate sphere, i.e. $\Sigma = S^{n-1} \times \{s\}$ for some number $s \in [s_0,\infty)$.
\end{theorem}

The inequality in Theorem \ref{thm.a} has a natural interpretation in terms of the Penrose inequality for collapsing null shells of dust. This is discussed in more detail in \cite{Brendle-Wang}.

We now discuss several results which are included in Theorem \ref{thm.a} as limiting cases. First, if we send $m \to 0$, then $s_0 \to 0$ and the AdS-Schwarzschild metric reduces to hyperbolic metric
\[\frac{1}{1 + s^2} \, ds \otimes ds + s^2 \, g_{S^{n-1}}.\]
Moreover, the static potential becomes $f = \sqrt{1+s^2} = \cosh r$, where $r$ denotes the geodesic distance from the origin. As a result, we obtain the following inequality for hypersurfaces in hyperbolic space:

\begin{theorem}
\label{thm.b}
Let $\Sigma$ be a compact mean convex hypersurface $\Sigma$ in the hyperbolic space $\mathbb{H}^n$ which is star-shaped with respect to the origin, and let $\Omega$ denote the region bounded by $\Sigma$. Then
\begin{align*}
&\int_\Sigma (f \, H - (n-1) \, \langle \bar{\nabla} f,\nu \rangle) \, d\mu \\
&\geq (n-1) \, |S^{n-1}|^{\frac{1}{n-1}} \, |\Sigma|^{\frac{n-2}{n-1}}.
\end{align*}
Moreover, equality holds if and only if $\Sigma$ is a geodesic sphere centered at the origin.
\end{theorem}

In particular, if the surface $\Sigma$ is very close to the origin, Theorem \ref{thm.b} reduces to the classical Minkowski inequality in $\mathbb{R}^n$. 

We next describe another limiting case of Theorem \ref{thm.a}. To that end, let us consider the rescaled metrics $m^{-\frac{2}{n-2}} \, \bar{g}$. After performing a change of variables, this metric can be written in the form  
\[\frac{1}{1 - s^{2-n} + m^{\frac{2}{n-2}} \, s^2} \, ds \otimes ds + s^2 \, g_{S^{n-1}}.\] 
If we send $m \to 0$, this metric converges to the standard Schwarzschild metric 
\[\frac{1}{1 - s^{2-n}} \, ds \otimes ds + s^2 \, g_{S^{n-1}},\] 
and the static potential $f$ converges to the static potential of the standard Schwarzschild manifold. Therefore, Theorem \ref{thm.a} implies a sharp Minkowski-type inequality for surfaces in the Schwarzschild manifold. 

The classical Minkowski inequality in $\mathbb{R}^n$ has important applications in general relativity, see \cite{Gibbons}. In particular, the total mean curvature integral appears in the definition of the Brown-York mass and Liu-Yau mass (cf. \cite{Liu-Yau1}, \cite{Liu-Yau2}). Our motivation came from the work \cite{Wang-Yau} in which a generalization of the positivity of Brown-York and Liu-Yau mass was considered when the reference space is a hyperbolic space. It was observed in \cite{Wang-Yau} that the mean curvature integral should be replaced by a weighted one in order to recover the right expression of mass (see \cite{Wang-Yau}, Theorem 1.4). The weighting factor is related to the coordinate functions of the embedding of a hyperboloid into the Minkowski space. The time component of the embedding can be chosen to be $\cosh r$ which is the same as the static potential here. In fact, the same weighting factor was considered in \cite{Shi-Tam} where another quasilocal mass with the hyperbolic space as reference was studied. We remark that the total mass for asymptotically hyperbolic manifolds has been considered by many authors, see e.g. \cite{Andersson-Dahl}, \cite{Chrusciel-Herzlich}, \cite{Chrusciel-Nagy}, \cite{Min-Oo}, \cite{Wang}, \cite{Zhang}. 

We now give an outline of the proof of Theorem \ref{thm.a}. An important tool in our proof is the inverse mean curvature flow. This method was employed in the spectacular proof of the Riemannian Penrose inequality in general relativity due to Huisken and Ilmanen \cite{Huisken-Ilmanen}. We start from a given mean convex, star-shaped hypersurface $\Sigma_0$, and evolve it by the inverse mean curvature flow. We show that the inverse mean curvature flow exists for all time, and that the evolving surfaces $\Sigma_t$ remain star-shaped for all $t \geq 0$. Moreover, we estimate the mean curvature and second fundamental form of $\Sigma_t$. More precisely, we prove that $|h_i^j-\delta_i^j| \leq O(t^2 \, e^{-\frac{2}{n-1} \, t})$. We note that the extra factor of $t^2$ can be removed, but we will not need this stronger estimate.

We next consider the quantity
\[Q(t) = |\Sigma_t|^{-\frac{n-2}{n-1}} \, \bigg ( \int_{\Sigma_t} f \, H \, d\mu - n(n-1) \int_\Omega f \, d\text{\rm vol} + (n-1) \, s_0^{n-2} \, |S^{n-1}| \bigg ),\]
where $f$ is the static potential defined above. It turns out that $Q(t)$ is monotone decreasing along the inverse mean curvature flow. The proof of this monotonicity property uses the fact that $(M,\bar{g})$ is static. We also use the inequality
\[(n-1) \int_{\Sigma_t} \frac{f}{H} \, d\mu \geq n \int_{\Omega_t} f \, d\text{\rm vol} + s_0^n \, |S^{n-1}|\]
(cf. \cite{Brendle}). This inequality was used in \cite{Brendle} to prove a generalization of Alexandrov's theorem (see also \cite{Brendle-Eichmair}).

Finally, we study the limit of $Q(t)$ as $t \to \infty$. The roundness estimate for $\Sigma_t$ is not strong enough to calculate the limit of $Q(t)$, and we expect that the limit of $Q(t)$ depends on the choice of the initial surface $\Sigma_0$. A similar issue arose in \cite{Neves}, where the limit of the Hawking mass was studied. However, we are able to give a lower bound for the limit of $Q(t)$. Using our estimate for the second fundamental form of $\Sigma_t$, we show that
\begin{align}
\label{asymptotics.Q}
Q(t) &\geq (n-1) \, \bigg ( \int_{S^{n-1}} \lambda^{n-1} \, d\text{\rm vol}_{S^{n-1}} \bigg )^{-\frac{n-2}{n-1}} \notag \\
&\cdot \bigg ( \frac{1}{2} \int_{S^{n-1}} \lambda^{n-4} \, |\nabla \lambda|_{g_{S^{n-1}}}^2 \, d\text{\rm vol}_{S^{n-1}} + \int_{S^{n-1}} \lambda^{n-2} \, d\text{\rm vol}_{S^{n-1}} \bigg ) - o(1),
\end{align}
where $\lambda$ is a positive function on $S^{n-1}$ which depends on $t$. In order to estimate the right hand side in \eqref{asymptotics.Q}, we use a sharp version of the Sobolev inequality on $S^{n-1}$ due to Beckner \cite{Beckner}. Using this inequality, we obtain
\[\liminf_{t \to \infty} Q(t) \geq (n-1) \, |S^{n-1}|^{\frac{1}{n-1}}.\]
Since $Q(t)$ is monotone decreasing, we conclude that $Q(0) \geq (n-1) \, |S^{n-1}|^{\frac{1}{n-1}}$. From this, Theorem \ref{thm.a} follows immediately.

We note that, after this paper was written, several related inequalities for hypersurfaces in hyperbolic space have appeared in the literature; see e.g. \cite{Lopes-de-Lima-Girao}, \cite{Wang-Xia}.

\section{Star-shaped hypersurfaces in the AdS-Schwarzschild manifold}

\begin{lemma}
By a change of variable, the AdS-Schwarzschild metric can be rewritten as
\[\bar{g} = dr \otimes dr+\lambda(r)^2 \, g_{S^{n-1}}\]
where $\lambda(r)$ satisfies the ODE 
\begin{equation}
\label{lambda'}
\lambda'(r)=\sqrt{1+\lambda^2-m\lambda^{2-n}}
\end{equation}
and the asymptotic expansion
\[\lambda(r) = \sinh(r) + \frac{m}{2n} \, \sinh^{1-n}(r) + O(\sinh^{-n-1}(r)).\]
\end{lemma}

\begin{proof}
Let us define a function $r(s)$ by 
\[r(s) = \int_0^s \frac{1}{\sqrt{1+t^2}} \, dt - \int_s^\infty \Big ( \frac{1}{\sqrt{1+t^2-mt^{2-n}}}-\frac{1}{\sqrt{1+t^2}} \Big ) \, dt.\] 
Moreover, let $\lambda$ be the inverse of the function $r(s)$, so that $\lambda(r(s)) = s$. With this understood, the metric $g$ can be written as $g = dr \otimes dr + \lambda(r)^2 \, g_{S^{n-1}}$. Furthermore, it is clear that $\lambda(r)$ satisfies the ODE \eqref{lambda'}. Finally, we have 
\begin{align*}
r(s) &= \int_0^s \frac{1}{\sqrt{1+t^2}} \, dt - \int_s^\infty \Big ( \frac{m}{2} \, t^{-n-1}+O(t^{-n-3}) \Big ) \, dt \\
&= \text{\rm arsinh}(s) - \frac{m}{2n} \, s^{-n} + O(s^{-n-2}).
\end{align*}
Hence, by Taylor expansion, we have
\begin{align*}
\sinh(r(s))
&= s - \frac{m}{2n} \, s^{1-n} + O(s^{-n-1}) \\
&= s - \frac{m}{2n} \, \sinh^{1-n}(r(s)) + O(\sinh^{-n-1}(r(s))).
\end{align*} From this, the assertion follows.
\end{proof}

We calculate the asymptotic expansion of Riemannian curvature tensors in the next lemma. Let $\theta=\{\theta^j\}_{j=1,2,\dots,n-1}$ be a local coordinate system on $S^{n-1}$ and let $\partial_{\theta^j}$ be the corresponding coordinate vector fields in $M$. 

\begin{lemma}
\label{curv_expansion}
Let $e_\alpha,\ \alpha=1,2,\dots,n$ be an orthonormal frame and let $R_{\alpha\beta\gamma\mu}$ denote the Riemannian curvature tensor of the AdS-Schwarzschild metric. Then
\begin{equation}
\label{riem.1}
R_{\alpha\beta\gamma\mu} = -\delta_{\beta\mu}\delta_{\alpha\gamma}+\delta_{\beta\gamma}\delta_{\alpha\mu}+O(e^{-nr})
\end{equation}
and
\begin{equation}
\label{riem.2}
\bar{\nabla}_\rho R_{\alpha\beta\gamma\mu}=O(e^{-nr}).
\end{equation}
Moreover, the Ricci tensor satisfies
\[\text{\rm Ric}(\partial_r,\partial_r)=-(n-1)-m \, \frac{(n-1)(n-2)}{2} \, \sinh^{-n}(r)+O(e^{-(n+2)r})\]
and
\[\lambda^{-2} \, \text{\rm Ric}(\partial_{\theta^i},\partial_{\theta^j})=\Big ( -(n-1)+m \, \frac{n-2}{2} \, \sinh^{-n}(r) \Big ) \, \sigma_{ij}+O(e^{-(n+2)r}),\]
where $\sigma_{ij} = g_{S^{n-1}}(\partial_{\theta^i},\partial_{\theta^j})$.
\end{lemma}

\begin{proof} 
Each level set of the function $r$ is a round sphere with induced metric $\lambda(r)^2 \, g_{S^{n-1}}$ and second fundamental form $\lambda(r) \, \lambda'(r) \, g_{S^{n-1}}$. Applying the Gauss equation, we compute
\[R(\partial_{\theta^i},\partial_{\theta^j},\partial_{\theta^k},\partial_{\theta^l}) =\lambda(r)^2 \, (1-\lambda'(r)^2) \, (\sigma_{ik}\sigma_{jl}-\sigma_{il}\sigma_{jk}).\]
Since each level set of $r$ is umbilic, from the Codazzi equation, we derive
\[R(\partial_{\theta^i},\partial_{\theta^j},\partial_{\theta^k},\partial_r)=0 .\]
The remaining components of the curvature tensors are
\begin{align*}
R(\partial_{\theta^i},\partial_r,\partial_{\theta^j},\partial_r)
&= \langle (\bar{\nabla}_{\partial_{\theta^i}} \bar{\nabla}_{\partial_r}-\bar{\nabla}_{\partial_r} \bar{\nabla}_{\partial_{\theta^i}})\partial_r,\partial_{\theta^j} \rangle \\
&= -\langle \bar{\nabla}_{\partial_r} \bar{\nabla}_{\partial_\theta^i} \partial_r,\partial_{\theta^j} \rangle \\
&= -\Big \langle \bar{\nabla}_{\partial_r} \Big ( \frac{\lambda'}{\lambda} \, \partial_{\theta^i} \Big ),\partial_{\theta^j} \Big \rangle \\
&=-\lambda(r) \, \lambda''(r) \, \sigma_{ij}.
\end{align*} From this, \eqref{riem.1} and \eqref{riem.2} follow easily.

Moreover, we have
\begin{align*} 
\text{\rm Ric}(\partial_r,\partial_r) 
&= -(n-1) \, \frac{\lambda''(r)}{\lambda(r)} \\ 
&= -(n-1) - m \, \frac{(n-1)(n-2)}{2} \, \sinh^{-n}(r)+O(e^{-(n+2)r}). 
\end{align*}
As the scalar curvature is equal to $-n(n-1)$, the expression of $\text{\rm Ric}(\partial_{\theta^i},\partial_{\theta^j})$ follows.
\end{proof}

A star-shaped hypersurface $\Sigma\subset M$ can be parametrized by
\[\Sigma = \{(r(\theta),\theta):\ \theta\in S^{n-1}\}\]
for a smooth function $r$ on $S^{n-1}$. We next define a new function $\varphi: S^{n-1} \to \mathbb{R}$ by
\[\varphi(\theta) = \Phi(r(\theta)),\]
where $\Phi(r)$ is a positive function satisfying $\Phi'(r) = \frac{1}{\lambda(r)}$.

Let $\varphi_i=\nabla_i\varphi$, $\varphi_{ij}=\nabla_j\nabla_i\varphi$, and $\varphi_{ijk}=\nabla_k\nabla_j\nabla_i \varphi$ denote the covariant derivatives of $\varphi$ with respect to the round metric $g_{S^{n-1}}$. Moreover, let
\begin{equation}
\label{def_rho}
\rho = \sqrt{1+|\nabla\varphi|_{S^{n-1}}^2}. 
\end{equation}
In the next lemma, we express the metric and second fundamental form of $\Sigma$ in terms of covariant derivatives of $\varphi$ as in \cite{Ding}:

\begin{proposition}
\label{induced_metric}
Let $g_{ij}$ be the induced metric on $\Sigma$ and $h_{ij}$ be the second fundamental form in term of the coordinates $\theta^j$. Then
\[g_{ij}=\lambda^2 \, (\sigma_{ij}+\varphi_i\varphi_j)\]
and
\[h_{ij}=\frac{\lambda}{\rho} \, \big ( \lambda' \, (\sigma_{ij}+\varphi_i\varphi_j) - \varphi_{ij} \big ).\]
\end{proposition}

\begin{proof}
A basis of tangent vector fields of $\Sigma$ is of the form $r_j\partial_r+\partial_{\theta^j}$. We compute
\begin{align*}
g_{ij} &= \langle r_i\partial_r+\partial_{\theta^i},r_j\partial_r+\partial_{\theta^j} \rangle \\
&= \lambda^2(r)\sigma_{ij}+r_ir_j \\
&=\lambda^2(r)(\sigma_{ij}+\varphi_i\varphi_j).
\end{align*}
The unit normal vector $\nu$ is given by
\[\nu = \frac{1}{\rho} \, \Big ( \partial_r-\frac{r^j}{\lambda^2}\partial_{\theta^j} \Big ). \]
Thus, the second fundamental form is given by
\begin{align*}
h_{ij}
&= -\big \langle \bar{\nabla}_{r_i\partial_r+\partial_{\theta^i}} (r_j\partial_r+\partial_{\theta^j}),\nu \big \rangle\\
&= -\Big \langle (r_{ij}-\lambda\lambda' \, \sigma_{ij}) \, \partial_r + \frac{\lambda'}{\lambda} \, r_j \, \partial_{\theta^i} + \frac{\lambda'}{\lambda} \, r_i \, \partial_{\theta^j},\nu \Big \rangle \\
&= \frac{1}{\rho} \, \Big ( \lambda\lambda' \, \sigma_{ij}+\frac{2\lambda'}{\lambda} \, r_i \, r_j-r_{ij} \Big )\\
&= \frac{\lambda}{\rho} \, \big ( \lambda' \, (\sigma_{ij}+\varphi_i\varphi_j)-\varphi_{ij} \big ),
\end{align*}
where $\bar{\nabla}$ denotes the Levi-Civita connection in the ambient AdS-Schwarzschild manifold.
\end{proof}

\section{The inverse mean curvature flow}

Let $\Sigma_0$ be a mean convex star-shaped hypersurface in $M$ which is given by an embedding
\[F_0:S^{n-1} \to M\]
Let $F_t:S^{n-1}\to M$, $t\in [0,T)$, be the solution of inverse mean curvature flow with initial data given by $F_0$. In other words,
\begin{equation}
\label{para_imcf}
\frac{\partial F}{\partial t} = \frac{1}{H} \, \nu,
\end{equation}
where $\nu$ is the unit outer normal vector and $H$ is the mean curvature. We shall call \eqref{para_imcf} the parametric form of the flow. 

We can write the initial hypersurface $\Sigma_0$ as the graph of a function $\tilde{r}_0$ defined on the unit sphere:
\[\Sigma_0=\{(\tilde{r}_0(\theta),\theta): \theta\in S^{n-1}\}.\]
If each $\Sigma_t$ is star-shaped, it can be parametrized them as the graph
\[\Sigma_t = \{(\tilde{r}(\theta,t),\theta): \theta\in S^{n-1}\}. \]
In this case, the inverse mean curvature flow can be written as a parabolic PDE for $\tilde{r}$. As long as the solution of \eqref{para_imcf} exists and remains star-shaped, it is equivalent to
\begin{equation}
\label{non_para_imcf}
\frac{\partial \tilde{r}}{\partial t} = \frac{\rho}{H},
\end{equation}
where $\rho$ is given by \eqref{def_rho}.

The equation \eqref{non_para_imcf} will be referred as the non-parametric form of the inverse mean curvature flow. Notice that the velocity vector of \eqref{para_imcf} is always normal, while the velocity vector of \eqref{non_para_imcf} is in the direction of $\partial_r$. To go from one to the other, we take the difference which is a (time-dependent) tangential vector field and compose the flow of the reparametrization associated with the tangent vector field.

Notice that associated with $\tilde{r}$, we define
\[\varphi(\theta,t) := \Phi(\tilde{r}(\theta,t)),\]
where $\Phi(r)$ is a positive function satisfying $\Phi'(r) = \frac{1}{\lambda(r)}$. Then $\varphi$ satisfies
\begin{equation}
\label{evol_varphi}
\frac{\partial \varphi}{\partial t} = \frac{\rho}{\lambda H}.
\end{equation}
In the sequel, we use the non-parametric form to derive $C^0$ and $C^1$ estimates of $\tilde{r}$. Some of theses estimates can be found in \cite{Ding} or \cite{Gerhardt2} (see also \cite{Gerhardt1}). For completeness, we derive all the estimates here.

\begin{lemma}
\label{bounds.for.lambda}
Let $\overline{r}(t)=\sup_{S^{n-1}} \tilde{r}(\cdot,t)$ and $\underline{r}(t) = \inf_{S^{n-1}} \tilde{r}(\cdot,t)$. Then
\[\lambda(\overline{r}(t)) \leq e^{\frac{1}{n-1} \, t} \, \lambda(\overline{r}(0))\]
and
\[\lambda(\underline{r}(t)) \geq e^{\frac{1}{n-1} \, t} \, \lambda(\underline{r}(0)).\]
\end{lemma}

\begin{proof}
Recall that
\[\frac{\partial \tilde{r}}{\partial t} = \frac{\rho}{H}.\]
Moreover, we have
\[H = \frac{(n-1)\lambda'}{\lambda \rho}-\frac{\tilde{\sigma}^{ij}}{\lambda \rho} \, \varphi_{ij},\]
where $\tilde{\sigma}^{ij}=\sigma^{ij}-\frac{\varphi^i\varphi^j}{\rho^2}$. At the point where the function $\tilde{r}(\cdot,t)$ attains its maximum, we have $H \geq \frac{(n-1)\lambda'}{\lambda}$. This implies
\[\frac{d}{dt} \overline{r}(t) \leq \frac{\lambda(\overline{r}(t))}{(n-1)\lambda'(\overline{r}(t))},\]
hence
\[\frac{d}{dt} \lambda(\overline{r}(t)) \leq \frac{\lambda(\overline{r}(t))}{n-1}.\] From this, the first statement follows. The second statement follows similarly.
\end{proof}

\begin{proposition}
\label{upper.bound.H}
We have $H \leq n-1+O(e^{-\frac{2}{n-1} \, t})$.
\end{proposition}

\begin{proof}
We work with the parametric formulation. The evolution of the mean curvature is given by
\[\frac{\partial H}{\partial t} = \frac{\Delta H}{H^2} - 2 \, \frac{|\nabla H|^2}{H^3} - \frac{|A|^2}{H} - \frac{\text{\rm Ric}(\nu,\nu)}{H}.\] 
Note that $|\text{\rm Ric} + (n-1) \, g| \leq O(e^{-\frac{n}{n-1} \, t})$ on $\Sigma_t$. This gives 
\begin{equation}
\label{eq_H}
\frac{\partial H}{\partial t} = \frac{\Delta H}{H^2} - 2 \, \frac{|\nabla H|^2}{H^3} - \frac{|A|^2}{H} + \frac{n-1}{H} + \frac{1}{H} \, O(e^{-\frac{n}{n-1} \, t}).
\end{equation}
Using \eqref{eq_H} and the inequality $|A|^2 \geq \frac{1}{n-1} \, H^2$, we obtain
\[\frac{d}{dt} H_{\text{\rm max}}^2 \leq -\frac{2}{n-1} \, H_{\text{\rm max}}^2 + 2(n-1) + O(e^{-\frac{n}{n-1} \, t}).\] This implies
\[H_{\text{\rm max}}(t)^2 \leq (n-1)^2 + O(e^{-\frac{2}{n-1} \, t}).\] From this, the assertion follows easily.
\end{proof}

We next establish a gradient bound for the function $\varphi$. For abbreviation, we define 
\[F = \frac{\lambda H}{\rho} = \frac{(n-1)\lambda'-\tilde{\sigma}^{ij}\varphi_{ij}}{\rho^2}\] 
and 
\[G_k = F \, \varphi_k-\frac{1}{\rho^2} \, \varphi^i \, \varphi_{ik} + \frac{1}{\rho^4} \, \varphi_k \, \varphi^i \, \varphi^j \, \varphi_{ij},\] 
where $\tilde{\sigma}^{ij}=\sigma^{ij}-\frac{\varphi^i\varphi^j}{\rho^2}$. Note that the variation of $F$ with respect to $\varphi_k$ is given by $-\frac{2}{\rho^2} \, G_k$.

\begin{proposition}
\label{grad.estimate.varphi}
We have $\sup_{S^{n-1}} |\nabla \varphi|_{g_{S^{n-1}}} = O(e^{-\frac{1}{n-1} \, t})$.
\end{proposition}

\begin{proof}
The non-parametric form of the equation can be written in the form 
\begin{equation}
\label{eq_phi}
\frac{\partial \varphi}{\partial t} = \frac{1}{F}.
\end{equation}
Let $\omega = \frac{1}{2} \, |\nabla\varphi|_{g_{S^{n-1}}}^2$. If we differentiate the identity \eqref{eq_phi} with respect to $\varphi^k \, \nabla_k$, we obtain
\begin{align*}
\frac{\partial\omega}{\partial t}
&= -\frac{1}{F^2} \, \varphi^k \, \nabla_k F \\ 
&= \frac{1}{\rho^2 F^2} \, \big ( \tilde{\sigma}^{ij} \, \varphi_{ijk} \, \varphi^k + 2 \, G^k \, \omega_k-2(n-1)\lambda\lambda'' \, \omega \big ).
\end{align*}
We next observe that
\begin{align*}
\omega_{ij}
&= \varphi_{kij} \, \varphi^k + \varphi_{ki} \, \varphi^k_{\ j} \\
&=\varphi_{ijk} \, \varphi^k + (\sigma_{ij}\sigma_{kp}-\sigma_{ik}\sigma_{jp}) \, \varphi^p \, \varphi^k+\varphi_{ki} \, \varphi^k_{\ j} \\
&=\varphi_{ijk} \, \varphi^k+\sigma_{ij} \, |\nabla\varphi|_{g_{S^{n-1}}}^2-\varphi_i \, \varphi_j+\varphi_{ki} \, \varphi^k_{\ j},
\end{align*}
where the covariant derivatives are taken with respect to $g_{S^{n-1}}$. Since
\[\tilde{\sigma}^{ij} \, (\sigma_{ij} \, |\nabla\varphi|_{g_{S^{n-1}}}^2 -\varphi_i\varphi_j) = 2(n-2) \, \omega,\] 
it follows that 
\[\tilde{\sigma}^{ij} \, \omega_{ij} = \tilde{\sigma}^{ij} \, \varphi_{ijk} + 2(n-2) \, \omega + \tilde{\sigma}^{ij} \, \varphi_{ki} \, \varphi^k_{\ j}.\] 
Putting these facts together, we conclude
\begin{align*}
\frac{\partial \omega}{\partial t} 
&= \frac{1}{\rho^2F^2} \, \big ( \tilde{\sigma}^{ij} \, \omega_{ij} + 2 \, G^k \, \omega_k - 2(n-2)\omega - 2(n-1)\lambda\lambda'' \, \omega \big ) \\ 
&- \frac{1}{\rho^2F^2} \, \tilde{\sigma}^{ij} \, \sigma^{kl} \, \varphi_{ik} \, \varphi_{jl} 
\end{align*}
Using Proposition \ref{upper.bound.H} and the inequality $\lambda'' > 0$, we obtain
\[\frac{(n-1)\lambda\lambda''}{\rho^2F^2} = \frac{(n-1)\lambda\lambda''}{\lambda^2H^2} \geq \frac{1}{n-1} - C \, e^{-\frac{2}{n-1} \, t}.\] 
Therefore, 
\[\frac{d}{dt} \omega_{\text{\rm max}} \leq -2 \, \Big ( \frac{1}{n-1} - C \, e^{-\frac{2}{n-1} \, t} \Big ) \, \omega_{\text{\rm max}},\]
where $\omega_{\text{\rm max}} = \frac{1}{2} \, \sup_{S^{n-1}} |\nabla\varphi|_{g_{S^{n-1}}}^2$. Thus $\omega_{\text{\rm max}}(t) = O(e^{-\frac{2}{n-1} \, t})$.
\end{proof}



\begin{proposition}
\label{time.derivative.of.varphi}
The function $\dot{\varphi} = \frac{\rho}{\lambda H}$ satisfies $\sup_{S^{n-1}} \dot{\varphi} \leq C \, e^{-\frac{1}{n-1} \, t}$.
\end{proposition}

\begin{proof}
If we differentiate \eqref{eq_phi} with respect to $t$, we obtain
\begin{align} 
\label{eq_dot_varphi}
\frac{\partial \dot{\varphi}}{\partial t} 
&= -\frac{1}{F^2} \, \frac{\partial F}{\partial t} \notag \\ 
&= \frac{1}{\rho^2F^2} \, \big ( \tilde{\sigma}^{ij} \, \dot{\varphi}_{ij} + 2 \, G^k \, \dot{\varphi}_k-(n-1)\lambda\lambda'' \, \dot{\varphi} \big ). 
\end{align}
As above, we have 
\[\frac{(n-1)\lambda\lambda''}{\rho^2F^2} = \frac{(n-1)\lambda\lambda''}{\lambda^2H^2} \geq \frac{1}{n-1} - C \, e^{-\frac{2}{n-1} \, t}\] 
in view of Proposition \ref{upper.bound.H}. Using the maximum principle, we obtain 
\[\sup_{S^{n-1}} \dot{\varphi} \leq C \, e^{-\frac{1}{n-1} \, t},\] 
as claimed.
\end{proof}

\begin{corollary} 
\label{lower.bound.for.H}
The mean curvature $H$ is bounded from below by a positive constant.
\end{corollary}

\begin{proof} 
By Proposition \ref{time.derivative.of.varphi}, we have $\frac{\rho}{\lambda H} \leq C \, e^{-\frac{1}{n-1} \, t}$ for some uniform constant $C$. Since $\rho \geq 1$ and $\lambda \leq C \, e^{\frac{1}{n-1} \, t}$, the assertion follows.
\end{proof}

Finally, we establish a uniform bound for the second fundamental form.

\begin{proposition}
\label{2nd.fund.form}
The norm of the second fundamental form is uniformly bounded globally in time.
\end{proposition}

\begin{proof}
We work with the parametric formulation. Using Lemma 2.1 in \cite{Huisken}, we compute
\begin{align*}
\frac{\partial h_i^j}{\partial t}
&= \frac{1}{H^2} \, \nabla^j \nabla_i H - 2 \, \frac{\nabla_i H \, \nabla^j H}{H^3} - \frac{h_i^kh_k^j}{H} - \frac{1}{H} \, g^{mj} \, R_{\nu i\nu m} \\ 
&= \frac{\Delta h_i^j}{H^2} - 2 \, \frac{\nabla_i H \, \nabla^j H}{H^3} + \frac{|A|^2}{H^2} \, h_i^j - 2 \, \frac{h_i^kh_k^j}{H} \\
&+ \frac{2}{H^2} \, g^{kl} \, g^{sj} \, R_{miks} \, h^m_l - \frac{1}{H^2} \, g^{kl} \, g^{sj} \, R_{mksl} \, h^m_i - \frac{1}{H^2} \, g^{kl} \, R_{mkil} \, h^{mj} \\
&+ \frac{1}{H^2} \, \text{\rm Ric}(\nu,\nu) \, h_i^j - \frac{2}{H} \, g^{mj} \, R_{\nu i\nu m} \\ 
&- \frac{1}{H^2} \, g^{kl} \, g^{mj} \, \bar{\nabla}_m R_{\nu kil} - \frac{1}{H^2} \, g^{kl} \, g^{mj} \, \bar{\nabla}_k R_{\nu iml}. 
\end{align*}
Using Lemma \ref{curv_expansion}, we obtain 
\begin{align}
\label{eq_second_fund}
\frac{\partial h_i^j}{\partial t}
&= \frac{\Delta h_i^j}{H^2} - 2 \, \frac{\nabla_i H \, \nabla^j H}{H^3} + \frac{|A|^2}{H^2} \, h_i^j - 2 \, \frac{h_i^kh_k^j}{H} \notag \\
&+ (n-1) \, \frac{h_i^j}{H^2} + \frac{|A|+1}{H^2} \, O(e^{-\frac{n}{n-1} \, t}). 
\end{align}
Combining \eqref{eq_H} and \eqref{eq_second_fund}, we obtain the following evolution equation for the tensor $M_i^j = H \, h_i^j$:
\begin{align}
\label{eq_M}
\frac{\partial M^j_i}{\partial t}
&= \frac{\Delta M^j_i}{H^2} - 2 \, \frac{\nabla^k H \, \nabla_k M^j_i}{H^3} - 2 \, \frac{\nabla_i H \, \nabla^j H}{H^2} \notag \\ 
&- 2 \, \frac{M^k_i \, M^j_k}{H^2} + 2(n-1) \, \frac{M^j_i}{H^2} + \frac{|M|+H}{H^2} \, O(e^{-\frac{n}{n-1} \, t}).
\end{align}
Let $\mu$ denote the largest eigenvalue of the tensor $M_i^j$, and let $\mu_{\text{\rm max}}(t)$ denote the maximum of $\mu$ at a given time $t$. Since the trace of $M$ is positive, we have $|M| \leq C\mu$ for some constant $C$. Since $H$ is uniformly bounded from above and below, we obtain 
\[\frac{d}{dt} \mu_{\text{\rm max}} \leq -\frac{1}{C} \, \mu_{\text{\rm max}}^2+C \, \mu_{\text{\rm max}}+C\]
for some uniform constant $C$. Therefore, $\mu_{\text{\rm max}} \leq C$ for some uniform constant $C$. This implies $|M| \leq C$. Since $H$ is uniformly bounded from below, we conclude that $|A|$ is uniformly bounded.
\end{proof}

\begin{corollary}
The solution of the inverse mean curvature flow is defined on $[0,\infty)$.
\end{corollary}

\section{The asymptotic behavior of the flow as $t \to \infty$}

In this section, we improved estimates for the mean curvature and second fundamental form. 

\begin{proposition}
We have $H=n-1+O(t \, e^{-\frac{2}{n-1} \, t})$.
\end{proposition}

\begin{proof}
In view of Proposition \ref{upper.bound.H}, it suffices to bound $H$ from below. To that end, we again work in the non-parametric setting. We consider the function 
\[\chi = \lambda \, \dot{\varphi} = \frac{\rho}{H}.\] 
The results in the previous section imply that the function $\chi$ is uniformly bounded from above and below. Using \eqref{eq_dot_varphi} and the identity $\dot{\varphi} = \frac{1}{F} = \frac{\chi}{\lambda}$, we obtain 
\begin{align*} 
\frac{\partial \chi}{\partial t} 
&= \lambda \, \frac{\partial \dot{\varphi}}{\partial t} + \lambda\lambda' \, \dot{\varphi}^2 \\ 
&= \frac{\chi^2}{\rho^2\lambda} \, \Big ( \tilde{\sigma}^{ij} \, \nabla_j \nabla_i \big ( \frac{\chi}{\lambda} \big ) + 2 \, G^k \, \nabla_k \big ( \frac{\chi}{\lambda} \big ) - (n-1)\lambda'' \, \chi \Big ) + \frac{\lambda'}{\lambda} \, \chi^2 \\ 
&= \frac{\chi^2}{\rho^2\lambda^2} \, \Big ( \tilde{\sigma}^{ij} \, \chi_{ij} - \frac{2}{\lambda} \, \tilde{\sigma}^{ij} \, \lambda_i \, \chi_j + 2 \, G^k \, \chi_k \Big ) \\ 
&+ \frac{\chi^2}{\rho^2\lambda^2} \, \Big ( \frac{2\chi}{\lambda^2} \, \tilde{\sigma}^{ij} \, \lambda_i \, \lambda_j - \frac{\chi}{\lambda} \, \tilde{\sigma}^{ij} \, \lambda_{ij} - \frac{2\chi}{\lambda} \, G^k \, \lambda_k \Big ) \\ 
&+ \frac{\lambda'}{\lambda} \, \chi^2 - \frac{n-1}{\rho^2} \, \frac{\lambda''}{\lambda} \, \chi^3. 
\end{align*} 
Using Proposition \ref{grad.estimate.varphi}, we obtain 
\[\tilde{\sigma}^{ij} \, \lambda_i \, \lambda_j \leq C \, e^{\frac{2}{n-1} \, t}.\] 
Moreover, using the identity 
\[-\tilde{\sigma}^{ij} \, \varphi_{ij} = \rho^2 \, F - (n-1) \, \lambda' = \rho^2 \, \frac{\lambda}{\chi} - (n-1) \, \lambda',\] 
we obtain 
\begin{align*} 
-\tilde{\sigma}^{ij} \, \lambda_{ij} 
&= -\lambda\lambda' \, \tilde{\sigma}^{ij} \, \varphi_{ij} - \lambda \, (\lambda\lambda'' + {\lambda'}^2) \, \tilde{\sigma}^{ij} \, \varphi_i \, \varphi_j \\ 
&\leq \lambda \lambda' \, \Big ( \rho^2 \, \frac{\lambda}{\chi} - (n-1) \, \lambda' \Big ) + C \, e^{\frac{1}{n-1} \, t}. 
\end{align*}
Finally, the second fundamental form is uniformly bounded by Proposition \ref{2nd.fund.form}. Using Proposition \ref{induced_metric}, we obtain $|D^2 \varphi| \leq C \, e^{\frac{1}{n-1} \, t}$, where $D^2 \varphi$ denotes the Hessian of $\varphi$ with respect to $g_{S^{n-1}}$. Using Proposition \ref{grad.estimate.varphi}, we conclude that 
\[-G^k \, \varphi_k = -F \, |\nabla\varphi|_{g_{S^{n-1}}}^2 + \frac{1}{\rho^4} \, \varphi^i \, \varphi^j \, \varphi_{ij} \leq C \, e^{-\frac{1}{n-1} \, t},\] 
hence 
\[-G^k \, \lambda_k \leq C \, e^{\frac{1}{n-1} \, t}.\] 
Putting these facts together, we conclude that 
\begin{align*} 
\frac{\partial \chi}{\partial t} 
&\leq \frac{\chi^2}{\rho^2\lambda^2} \, \Big ( \tilde{\sigma}^{ij} \, \chi_{ij} - \frac{2}{\lambda} \, \tilde{\sigma}^{ij} \, \lambda_i \, \chi_j + 2 \, G^k \, \chi_k \Big ) \\ 
&+ \frac{2\lambda'}{\lambda} \, \chi^2 - \frac{n-1}{\rho^2} \, \frac{\lambda \, \lambda''+{\lambda'}^2}{\lambda^2} \, \chi^3 + C \, e^{-\frac{2}{n-1} \, t}. 
\end{align*} 
Since $\chi$ is uniformly bounded and $\rho = 1 + O(e^{-\frac{2}{n-1} \, t})$, the maximum of $\chi$ satisfies 
\[\frac{d}{dt} \chi_{\text{\rm max}} \leq 2 \, \chi_{\text{\rm max}}^2 - 2(n-1) \, \chi_{\text{\rm max}}^3 + C \, e^{-\frac{2}{n-1} \, t}.\] 
In particular, we have 
\[\frac{d}{dt} \chi_{\text{\rm max}} \leq \frac{2}{(n-1)^2} - \frac{2}{n-1} \, \chi_{\text{\rm max}} + C \, e^{-\frac{2}{n-1} \, t}\] 
whenever $\chi_{\text{\rm max}} \geq \frac{1}{n-1}$. Therefore, $\chi_{\text{\rm max}} \leq \frac{1}{n-1} + O(t \, e^{-\frac{2}{n-1} \, t})$. Since $\rho = 1 + O(e^{-\frac{2}{n-1} \, t})$, we conclude that $H \geq n-1 - O(t \, e^{-\frac{2}{n-1} \, t})$. 
\end{proof}

\begin{proposition}
\label{2nd.fund.form.2}
We have $|h_i^j-\delta_i^j| \leq O(t^2 \, e^{-\frac{2}{n-1} \, t})$.
\end{proposition}

\begin{proof}
As above, we define $M_i^j = H \, h_i^j$. We have shown above that $|M|$ is uniformly bounded, and $H=n-1+O(t \, e^{-\frac{2}{n-1} \, t})$. Hence, it follows from \eqref{eq_M} that 
\begin{align*}
\frac{\partial M^j_i}{\partial t}
&= \frac{\Delta M^j_i}{H^2} - 2 \, \frac{\nabla^k H \, \nabla_k M^j_i}{H^3} - 2 \, \frac{\nabla_i H \, \nabla^j H}{H^2} \\ 
&- \frac{2}{(n-1)^2} \, M^k_i \, M^j_k + \frac{2}{n-1} \, M^j_i + O(t \, e^{-\frac{n}{n-1} \, t}). 
\end{align*}
Let $\mu$ denote the largest eigenvalue of $M_i^j$, and let $\mu_{\text{\rm max}}(t)$ be the maximum of $\mu$ at a given time $t$. Then 
\begin{align*} 
\frac{d}{dt} \mu_{\text{\rm max}} 
&\leq -\frac{2}{(n-1)^2} \, \mu_{\text{\rm max}}^2 + \frac{2}{n-1} \, \mu_{\text{\rm max}} + O(t \, e^{-\frac{n}{n-1} \, t}) \\ 
&\leq 2-\frac{2}{n-1} \, \mu_{\text{\rm max}}+O(t \, e^{-\frac{2}{n-1} \, t}). 
\end{align*}
Thus, 
\[\mu_{\text{\rm max}} \leq n-1+O(t^2 \, e^{-\frac{2}{n-1} \, t}).\] 
As $M_i^j=Hh_i^j$ and $H=n-1+O(t \, e^{-\frac{2}{n-1} \, t})$, we conclude that the largest eigenvalue of the second fundamental form is less than $1+O(t^2 \, e^{-\frac{2}{n-1} \, t})$. Since $H=n-1+O(t \, e^{-\frac{2}{n-1} \, t})$, the smallest eigenvalue of the second fundamental form is greater than $1-O(t^2 \, e^{-\frac{2}{n-1} \, t})$.
\end{proof}

\section{The monotonicity formula}

As above, we consider a family of star-shaped surfaces $\Sigma_t$ evolving by inverse mean curvature flow. We define
\[Q(t) = |\Sigma_t|^{-\frac{n-2}{n-1}} \, \bigg ( \int_{\Sigma_t} f \, H \, d\mu - n(n-1) \int_\Omega f \, d\text{\rm vol} + (n-1) \, s_0^{n-2} \, |S^{n-1}| \bigg ),\]
where $f = \sqrt{1+\lambda^2-m \, \lambda^{2-n}}$ denotes the static potential.

We first evaluate the limit of $Q(t)$ as $t \to \infty$. To that end, we need the following auxiliary result:

\begin{proposition}
\label{sharp_sobolev}
For every positive function $u$ on $S^{n-1}$, we have
\begin{align*}
&\frac{1}{2} \int_{S^{n-1}} u^{n-4} \, |\nabla u|_{g_{S^{n-1}}}^2 \, d\text{\rm vol}_{S^{n-1}} + \int_{S^{n-1}} u^{n-2} \, d\text{\rm vol}_{S^{n-1}} \\
&\geq |S^{n-1}|^{\frac{1}{n-1}} \, \bigg ( \int_{S^{n-1}} u^{n-1} \, d\text{\rm vol}_{S^{n-1}} \bigg )^{\frac{n-2}{n-1}}.
\end{align*}
Moreover, equality holds if and only if $u$ is constant.
\end{proposition}

\begin{proof}
It follows from Theorem 4 in \cite{Beckner} that
\begin{align*}
&\frac{2}{(n-2)(n-1)} \int_{S^{n-1}} |\nabla v|_{g_{S^{n-1}}}^2 \, d\text{\rm vol}_{S^{n-1}} + \int_{S^{n-1}} v^2 \, d\text{\rm vol}_{S^{n-1}} \\
&\geq |S^{n-1}|^{\frac{1}{n-1}} \, \bigg ( \int_{S^{n-1}} v^{\frac{2(n-1)}{n-2}} \, d\text{\rm vol}_{S^{n-1}} \bigg )^{\frac{n-2}{n-1}}
\end{align*}
for every positive smooth function $v$. Hence, if we put $v = u^{\frac{n-2}{2}}$, we obtain
\begin{align*}
&\frac{n-2}{2(n-1)} \int_{S^{n-1}} u^{n-4} \, |\nabla u|_{g_{S^{n-1}}}^2 \, d\text{\rm vol}_{S^{n-1}} + \int_{S^{n-1}} u^{n-2} \, d\text{\rm vol}_{S^{n-1}} \\
&\geq |S^{n-1}|^{\frac{1}{n-1}} \, \bigg ( \int_{S^{n-1}} u^{n-1} \, d\text{\rm vol}_{S^{n-1}} \bigg )^{\frac{n-2}{n-1}}.
\end{align*} From this, the assertion follows. \end{proof}

\begin{proposition}
We have $\liminf_{t \to \infty} Q(t) \geq (n-1) \, |S^{n-1}|^{\frac{1}{n-1}}$.
\end{proposition}

\begin{proof}
Using the inequalities
\begin{align*}
&f = \lambda \, (1 + O(e^{-\frac{2}{n-1} \, t})), \\
&H-n+1 = O(t \, e^{-\frac{2}{n-1} \, t}), \\
&\sqrt{\det g} = \lambda^{n-1} \, \sqrt{\det g_{S^{n-1}}} \, (1 + O(e^{-\frac{2}{n-1} \, t})),
\end{align*}
we obtain
\begin{equation}
\label{step.1}
\int_{\Sigma_t} f \, (H - n+1) \, d\mu = \int_{S^{n-1}} \lambda^n \, (H - n+1) \, d\text{\rm vol}_{S^{n-1}} + O(t \, e^{\frac{n-4}{n-1} \, t}).
\end{equation}
By Proposition \ref{induced_metric}, the metric and second fundamental form on $\Sigma_t$ are given by
\[g_{ij} = \lambda^2 \, (\sigma_{ij}+\varphi_i\varphi_j)\]
and
\[h_{ij} = \frac{\lambda'}{\lambda \rho} \, g_{ij} - \frac{\lambda}{\rho} \, \varphi_{ij}.\]
Here, $\sigma_{ij}$ is the round metric on $S^{n-1}$ and $\varphi_{ij}$ is the Hessian of $\varphi$ with respect to $g_{S^{n-1}}$. By Proposition \ref{2nd.fund.form.2}, we have $|h-g|_g \leq O(t^2 \, e^{-\frac{2}{n-1} \, t})$. This implies
\[\Big | h - \frac{\lambda'}{\lambda \rho} \, g \Big |_g \leq O(t^2 \, e^{-\frac{2}{n-1} \, t}),\]
hence
\[\Big | h - \frac{\lambda'}{\lambda \rho} \, g \Big |_{g_{S^{n-1}}} \leq O(t^2).\] From this, we deduce that $|D^2 \varphi|_{g_{S^{n-1}}} \leq O(t^2 \, e^{-\frac{1}{n-1} \, t})$, where $D^2 \varphi$ denotes the Hessian of $\varphi$ with respect to $g_{S^{n-1}}$. Using Proposition \ref{grad.estimate.varphi}, we obtain
\[\tilde{\sigma}^{ij} \, \varphi_{ij} = \Delta_{S^{n-1}} \varphi + O(t^2 \, e^{-\frac{3}{n-1} \, t}).\]
This implies
\begin{align*}
H &= \frac{(n-1) \lambda'}{\lambda \rho} - \frac{1}{\lambda \rho} \, \tilde{\sigma}^{ij} \, \varphi_{ij} \\
&= \frac{(n-1) \lambda'}{\lambda \rho} - \frac{1}{\lambda \rho} \, \Delta_{S^{n-1}} \varphi + O(t^2 \, e^{-\frac{4}{n-1} \, t}).
\end{align*}
Since $\lambda' = \lambda + \frac{1}{2} \, \lambda^{-1} + O(e^{-\frac{2}{n-1} \, t})$ and $\frac{1}{\rho} = 1 - \frac{1}{2} \, |\nabla \varphi|_{g_{S^{n-1}}}^2 + O(e^{-\frac{4}{n-1} \, t})$, we conclude that
\[H = n-1 + \frac{n-1}{2\lambda^2} - \frac{n-1}{2} \, |\nabla \varphi|_{g_{S^{n-1}}}^2 - \frac{1}{\lambda} \, \Delta_{S^{n-1}} \varphi + O(e^{-\frac{3}{n-1} \, t}).\]
Substituting this identity into \eqref{step.1}, we obtain
\begin{align*}
&\int_{\Sigma_t} f \, (H - n+1) \, d\mu \\
&= \int_{S^{n-1}} \Big ( \frac{n-1}{2} \, \lambda^{n-2} - \frac{n-1}{2} \, \lambda^n \, |\nabla \varphi|_{g_{S^{n-1}}}^2 - \lambda^{n-1} \, \Delta_{S^{n-1}} \varphi \Big ) \, d\text{\rm vol}_{S^{n-1}} + O(e^{\frac{n-3}{n-1} \, t}) \\
&= \int_{S^{n-1}} \Big ( \frac{n-1}{2} \, \lambda^{n-2} - \frac{n-1}{2} \, \lambda^n \, |\nabla \varphi|_{g_{S^{n-1}}}^2 + (n-1) \, \lambda^{n-2} \, \langle \nabla \lambda,\nabla \varphi \rangle_{S^{n-1}} \Big ) \, d\text{\rm vol}_{S^{n-1}} \\
&+ O(e^{\frac{n-3}{n-1} \, t}).
\end{align*}
By Proposition \ref{grad.estimate.varphi}, we have $|\nabla \varphi|_{g_{S^{n-1}}} \leq O(e^{-\frac{1}{n-1}} \, t)$. Since $\nabla \lambda = \lambda\lambda' \, \nabla \varphi$, it follows that $|\nabla \lambda - \lambda^2 \, \nabla \varphi|_{g_{S^{n-1}}} \leq O(e^{-\frac{1}{n-1}} \, t)$. This implies
\begin{align}
\label{step.2}
&\int_{\Sigma_t} f \, (H - n+1) \, d\mu \notag \\ 
&= \int_{S^{n-1}} \Big ( \frac{n-1}{2} \, \lambda^{n-2} + \frac{n-1}{2} \, \lambda^{n-4} \, |\nabla \lambda|_{g_{S^{n-1}}}^2 \Big ) \, d\text{\rm vol}_{S^{n-1}} + O(e^{\frac{n-3}{n-1} \, t}).
\end{align}
On the other hand, the static potential satisfies
\begin{align*}
f - \langle \bar{\nabla} f,\nu \rangle
&\geq f-|\bar{\nabla} f| \\
&= \sqrt{1+\lambda^2-m\lambda^{2-n}}-(\lambda+\frac{m(n-2)}{2}\lambda^{1-n}) \\
&\geq \frac{1}{2} \, \lambda^{-1} - O(\lambda^{-2}). 
\end{align*}
This gives
\begin{equation}
\label{step.3}
(n-1) \int_{\Sigma_t} (f - \langle \bar{\nabla} f,\nu \rangle) \, d\mu \geq \frac{n-1}{2} \int_{S^{n-1}} \lambda^{n-2} \, d\text{\rm vol}_{S^{n-1}} - O(e^{\frac{n-3}{n-1} \, t}).
\end{equation}
Moreover, using the identity $\bar{\Delta} f = nf$ and the divergence theorem, we obtain 
\begin{equation} 
\label{step.4}
(n-1) \int_{\Sigma_t} \langle \bar{\nabla} f,\nu \rangle \, d\mu - n(n-1) \int_{\Omega_t} f \, d\text{\rm vol} = O(1). 
\end{equation}
Adding \eqref{step.2}, \eqref{step.3}, and \eqref{step.4}, we obtain
\begin{align*}
&\int_{\Sigma_t} f \, H \, d\mu - n(n-1) \, \int_{\Omega_t} f \, d\text{\rm vol} \\
&\geq \frac{n-1}{2} \int_{S^{n-1}} \lambda^{n-4} \, |\nabla \lambda|_{g_{S^{n-1}}}^2 \, d\text{\rm vol}_{S^{n-1}} \\
&+ (n-1) \int_{S^{n-1}} \lambda^{n-2} \, d\text{\rm vol}_{S^{n-1}} - O(e^{\frac{n-3}{n-1} \, t}).
\end{align*}
Moreover,
\[|\Sigma_t| = \int_{S^{n-1}} \lambda^{n-1} \, d\text{\rm vol}_{S^{n-1}} + O(e^{\frac{n-3}{n-1} \, t}).\]
Using Proposition \ref{sharp_sobolev}, we conclude that
\[\liminf_{t\to \infty} |\Sigma_t|^{-\frac{n-2}{n-1}} \, \bigg ( \int_{\Sigma_t} f \, H \, d\mu - n(n-1) \int_{\Omega_t} f \, d\text{\rm vol} \bigg ) \geq (n-1) \, |S^{n-1}|^{\frac{1}{n-1}}.\]
This completes the proof.
\end{proof}

Finally, we show that $Q(t)$ is monotone along the flow:

\begin{proposition}
The quantity $Q(t)$ is monotone decreasing in $t$.
\end{proposition}

\begin{proof}
The evolution of the mean curvature is given by
\[\frac{\partial}{\partial t} H = -\Delta \Big ( \frac{1}{H} \Big ) - \frac{1}{H} \,  (|A|^2 + \text{\rm \text{\rm Ric}}(\nu,\nu)).\]
This implies
\[\frac{\partial}{\partial t} (f \, H) = -f \, \Delta \Big ( \frac{1}{H} \Big ) - \frac{f}{H} \,  (|A|^2 + \text{\rm \text{\rm Ric}}(\nu,\nu)) + \langle \bar{\nabla} f,\nu \rangle.\]
Using the identity $\Delta f = \bar{\Delta} f - (D^2 f)(\nu,\nu) - H \, \langle \bar{\nabla} f,\nu \rangle$, we obtain
\begin{align}
\label{calculation}
\frac{d}{dt} \bigg ( \int_{\Sigma_t} f \, H \, d\mu \bigg ) 
&= -\int_{\Sigma_t} f \, \Delta \Big ( \frac{1}{H} \Big ) \, d\mu - \int_{\Sigma_t} \frac{f}{H} \, (|A|^2 + \text{\rm \text{\rm Ric}}(\nu,\nu)) \, d\mu \notag \\
&+ \int_{\Sigma_t} (\langle \bar{\nabla} f,\nu \rangle + f \, H) \, d\mu \notag \\
&= -\int_{\Sigma_t} \frac{1}{H} \, \Delta f \, d\mu - \int_{\Sigma_t} \frac{f}{H} \, (|A|^2 + \text{\rm Ric}(\nu,\nu)) \, d\mu \notag \\
&+ \int_{\Sigma_t} (\langle \bar{\nabla} f,\nu \rangle + f \, H) \, d\mu \notag \\
&= -\int_{\Sigma_t} \frac{1}{H} \, (\bar{\Delta} f - (D^2 f)(\nu,\nu)) \, d\mu \\ 
&- \int_{\Sigma_t} \frac{f}{H} \, (|A|^2 + \text{\rm Ric}(\nu,\nu)) \, d\mu \notag \\
&+ \int_{\Sigma_t} (2 \, \langle \bar{\nabla} f,\nu \rangle + f \, H) \, d\mu \notag \\
&= -\int_{\Sigma_t} \frac{f}{H} \, |A|^2 + \int_{\Sigma_t} (2 \, \langle \bar{\nabla} f,\nu \rangle + f \, H) \, d\mu \notag \\
&\leq \int_{\Sigma_t} \Big ( 2 \, \langle \bar{\nabla} f,\nu \rangle + \frac{n-2}{n-1} \, f \, H \Big ) \, d\mu. \notag
\end{align}
Using the identity $\bar{\Delta} f = nf$ and the divergence theorem, we obtain
\[\int_{\Sigma_t} \langle \bar{\nabla} f,\nu \rangle \, d\mu = n \int_{\Omega_t} f \, d\text{\rm vol} + \frac{(n-2)m + 2s_0^n}{2} \, |S^{n-1}|.\] 
Moreover, it was shown in \cite{Brendle} that
\[(n-1) \int_{\Sigma_t} \frac{f}{H} \, d\mu \geq n \int_{\Omega_t} f \, d\text{\rm vol} + s_0^n \, |S^{n-1}|.\] 
Putting these facts together, we conclude that
\begin{align*}
&\frac{d}{dt} \bigg ( \int_{\Sigma_t} f \, H \, d\mu - n(n-1) \int_{\Omega_t} f \, d\text{\rm vol} \bigg ) \\
&\leq \int_{\Sigma_t} \Big ( 2 \, \langle \bar{\nabla} f,\nu \rangle \, d\mu + \frac{n-2}{n-1} \, f \, H - n(n-1) \, \frac{f}{H} \Big ) \, d\mu \\
&\leq \frac{n-2}{n-1} \int_{\Sigma_t} f \, H \, d\mu - n(n-2) \int_{\Omega_t} f \, d\text{\rm vol} \\
&+ ((n-2)m + 2 s_0^n) \, |S^{n-1}| - n \, s_0^n \, |S^{n-1}| \\
&= \frac{n-2}{n-1} \, \bigg ( \int_{\Sigma_t} f \, H \, d\mu - n(n-1) \int_{\Omega_t} f \, d\text{\rm vol} + (n-1) \, s_0^{n-2} \, |S^{n-1}| \bigg ).
\end{align*}
Thus, we conclude that $\frac{d}{dt} Q(t) \leq 0$, and equality holds when the surfaces $\Sigma_t$ are coordinate spheres.
\end{proof}

\begin{corollary}
We have 
\begin{align*}
&\int_{\Sigma_0} f \, H \, d\mu - n(n-1) \int_{\Omega_0} f \, d\text{\rm vol} \\
&\geq (n-1) \, |S^{n-1}|^{\frac{1}{n-1}} \, \big ( |\Sigma_0|^{\frac{n-2}{n-1}} - |\partial M|^{\frac{n-2}{n-1}} \big ).
\end{align*}
\end{corollary}

\begin{proof} 
Since $Q(t)$ is monotone decreasing, we have 
\[Q(0) \geq \liminf_{t \to \infty} Q(t) \geq (n-1) \, |S^{n-1}|^{\frac{1}{n-1}}.\] 
This implies 
\begin{align*}
&\int_{\Sigma_0} f \, H \, d\mu - n(n-1) \int_{\Omega_0} f \, d\text{\rm vol} \\
&\geq (n-1) \, |S^{n-1}|^{\frac{1}{n-1}} \, |\Sigma_0|^{\frac{n-2}{n-1}} - (n-1) \, s_0^{n-2} \, |S^{n-1}|.
\end{align*}
Since $|\partial M| = s_0^{n-1} \, |S^{n-1}|$, the assertion follows.
\end{proof}

It remains to discuss the case of equality. Suppose that 
\begin{align*}
&\int_{\Sigma_0} f \, H \, d\mu - n(n-1) \int_{\Omega_0} f \, d\text{\rm vol} \\
&= (n-1) \, |S^{n-1}|^{\frac{1}{n-1}} \, \big ( |\Sigma_0|^{\frac{n-2}{n-1}} - |\partial M|^{\frac{n-2}{n-1}} \big ).
\end{align*} 
In this case, the function $Q(t)$ is constant. In particular, we must have equality in \eqref{calculation}. Consequently, the surface $\Sigma_0$ is umbilic. If the mass $m$ is positive, it follows that $\Sigma_0$ is a coordinate sphere, as claimed. On the other hand, if the mass $m$ vanishes, then $\Sigma_0$ must be a geodesic sphere centered at some point $x_0$. If $x_0$ is not the origin, then the function $\lambda$ converges to a non-constant function on $S^{n-1}$ after rescaling. Using the equality statement in Proposition \ref{sharp_sobolev}, we conclude that $\liminf_{t \to \infty} Q(t) > (n-1) \, |S^{n-1}|^{\frac{1}{n-1}}$, contrary to our assumption. Thus, $\Sigma_0$ must be a geodesic sphere centered at the origin.

\end{document}